\documentclass[preprint,12pt,3p]{elsarticle}
\usepackage{amsmath,amsthm,amsfonts,amssymb}
\usepackage{fullpage}
\usepackage{subfig}
\usepackage{blkarray}
\usepackage{caption}
\usepackage{float}
\usepackage{algpseudocode}
\usepackage{tikz}
\usepackage{hyperref}
\numberwithin{equation}{section}
\usepackage{xcolor}
\usepackage{colortbl}
\usepackage[normalem]{ulem}
\usepackage{sectsty}
\definecolor{astral}{RGB}{46,116,181}
\subsectionfont{\color{astral}}
\sectionfont{\color{astral}}
\usepackage{colortbl}

\DeclareMathAlphabet{\mathpzc}{OT1}{pzc}{m}{it}
\DeclareFontFamily{OT1}{pzc}{}
\DeclareFontShape{OT1}{pzc}{m}{it}{<-> s * [0.900] pzcmi7t}{}
\DeclareMathAlphabet{\mathpzc}{OT1}{pzc}{m}{it}

\usepackage{amsbsy}
\usepackage{amsmath}
\usepackage{accents}
\newlength{\dhatheight}

\newenvironment{frcseries}{\fontfamily{frc}\selectfont}{}

\DeclareMathAlphabet\mathbfcal{OMS}{cmsy}{b}{n}
\newcommand{\core}[1]{#1^{\tiny{\textcircled{\tiny\#}}}}

\newcommand{\dc}[1]{#1e,{\tiny{\textcircled{\tiny {\#}}}}}
\definecolor{darkslategray}{rgb}{0.18, 0.31, 0.31}
\definecolor{warmblack}{rgb}{0.0, 0.26, 0.26}

\usepackage{multirow}
\usepackage{mathtools}
\usepackage{algorithm}
\usepackage{algorithmicx}
\makeatletter
\def\BState{\State\hskip-\ALG@thistlm}
\makeatother

\newtheorem{theorem}{Theorem}[section]
\newtheorem{lemma}[theorem]{Lemma}

\theoremstyle{definition}
\usepackage{hyperref}

\newtheorem{remark}{Remark}[section]

\journal{ Communications in Algebra}

\begin{document}

\begin{frontmatter}

\title{On forward-order law for core inverse in rings}

\vspace{-.4cm}

\author{Amit Kumar$^b$ and Debasisha Mishra$^{*b}$}

 \address{

                          $^b$Department of Mathematics,\\
                        National Institute of Technology Raipur, India.
                        \\email: amitdhull513@gmail.com  \\
                        email$^*$: dmishra\symbol{'100}nitrr.ac.in. \\}
\vspace{-2cm}

\begin{abstract}
If $a$ and $b$ are a pair of invertible elements, then $ab$ is also invertible and the inverse
of the product $ab$ satisfying 
  $$(ab)^{-1}=a^{-1}b^{-1}$$ is known as the {\it forward-order law}.
 This article establishes a few sufficient conditions of the forward-order law for the core inverse of elements in rings with involution. It also presents the forward-order law for the weighted core inverse and the triple forward-order law  for the core inverse. Additionally,  we discuss the hybrid forward-order law among the Moore-Penrose inverse, the group inverse, and the core inverse.
 \end{abstract}


\end{frontmatter}

\newpage
\section{Introduction}
Inverses/generalized inverses of a product of two or three elements over a ring were
investigated by many researchers.  For instance, if $a$ and $b$ are a pair of invertible elements, then $ab$ is also invertible, and the inverse
of the product $ab$ satisfying 
 $$(ab)^{-1}=b^{-1}a^{-1},$$
 is known as the {\it reverse-order law}. On the other way, $$(ab)^{-1}=a^{-1}b^{-1}$$ is known as the {\it forward-order law}.
While the reverse-order law does not hold for different generalized inverses (see the next section for its definition),  the forward-order law is not valid even for invertible elements. One of the fundamental topics in the theory of generalized inverses is investigating various reverse-order laws and forward-order laws. For example, Mosic and Djordjevic \cite{Mosic7} expanded the reverse-order law for the Moore-Penrose inverse of a matrix to the reverse-order law for the Moore-Penrose inverse of an element over ring. In 2012, Mosic and Djordjevic \cite{Mosic} extended the reverse-order law for the  group inverse in Hilbert space to ring.   
 In 2017, Zhu and Chen \cite{Zhu25} bestowed the forward-order law for the Drazin inverse in a ring. Zhu (\cite{Chen20} and \cite{Chen22}) conferred several results on  additive properties,  reverse-order law and forward-order law. In 2012, Mosic and Djordjevic  \cite{Mosic17} provided the hybrid reverse-order law between the group inverse and the Moore-Penrose inverse. 
In 2018, Zhu {\it et al.} \cite{Chen14}  provided the reverse-order law for the generalized core inverse. In 2020, Sahoo {\it et al.} \cite{Sahoo} endowed the reverse-order law of the weighted core inverse. In 2021,  Gao {\it et al.} \cite{Gao3} provided the reverse-order law for the generalized pseudo core inverse.  In 2021, Li {\it et al.} \cite{li} studied the forward-order law for the  core inverse and the hybrid forward-order law among the  core inverse, the Moore-Penrose inverse and the group inverse in matrix setting.
 The vast literature on the core inverse and the weighted core inverse with its multifarious extensions in different areas of mathematics motivate us to study the following two problems.
 \begin{center}
 \begin{itemize}
     \item[(i)]   When does the forward-order law for the core inverse and the weighted core inverse hold?
     \item[(ii)] When does the triple forward-order law hold?
 \end{itemize}

\end{center}
 In application point of view, the core inverse is used to find the Bott-Duffin inverse \cite{butt}. Also,  one can compute the Moore-Penrose inverse and the group inverse by using the forward-order law for the core inverse. The theory of generalized inverses over rings is also used in cryptography \cite{dra}. For example: To find the solution of  $y=ax+b$(mod 26) which is $x=a^{-1}(y-b)$(mod 26) where  $a^{-1}$ is a generalized inverse of $a$
 \cite{bokr}. 
  The original idea of introducing the matrix partial orders comes from the partial orders that were defined in the
context of semigroups. It is thus interesting to study the above-stated problems in rings. \par
 
 The main goal of this article is, therefore, to study the forward-order law for the core inverse and the weighted core inverse. Furthermore, examples that stress their importance are presented, and counterexamples  are also illustrated. Before we begin, we organize this article as follows. Section \ref{sec2} presents a few conditions for the forward-order law for the core inverse. Some conditions for the  forward-order law for the weighted core inverse are also established. Finally, the triple forward-order law is discussed for the core inverse. Section \ref{sec3} provides some necessary and sufficient conditions of the hybrid forward-order law. 
\section{Preliminaries}
Throughout this article,  $R$ denotes a unital $*$-ring, that is, a ring with unity 1 and $*$ an involution. A ring $R$ is called as a {\it proper ring} if $a^*a=0\Rightarrow a=0$ for all $a\in R$.  An {\it involution $*$} is an  anti-isomorphism of order 2 that satisfies the conditions 
$$(a+b)^*=a^*+b^*,~~(ab)^*=b^*a^*, \text{ and } (a^*)^*=a, \text{ for all }a, b \in R.$$
An element $a$  is said to be {\it Hermitian} if $a^*=a$, and  is called  {\it idempotent}  if $a^2=a.$ The left annihilator of $a\in R$ is given by ${^{\circ}{(a)}}=\{x\in R:xa=0\} $ and  the right  annihilator of $a$ is given by $(a)^{\circ}=\{x\in R:ax=0\}$. An element $a\in R$ is {\it Moore-Penrose invertible} if there exists a unique element $x\in R$ that satisfies the equations:
$$(1.)~ axa=a, ~~(2.)~ xax=x, ~~(3.)~ (ax)^*=ax, \textnormal{ and } (4.)~ (xa)^*=xa.$$
Then, $x$ is called as the {\it Moore-Penrose inverse} \cite{Mpenrose}  of $a$, and is denoted as $x=a^{\dagger}$. By $R^{\dagger}$, we denote the set of all Moore-Penrose invertible elements of $R$. The set of all elements which satisfies any of the combinations of the above four equations is denoted as $a\{i,j,k,l\}$, where $i,j,k,l \in \{1,2,3,4\}$, and is called a generalized inverse of $a$. The first and third generalized inverse of $a$ is denoted as $a^{(1,3)}.$ The set of first and third invertible elements of $R$, is denoted by $R^{(1,3)}.$
An element $a $ is called  {\it Drazin invertible} \cite{Drazin}  if  there exists a unique element $x\in R$ such that
$xa^{k+1}=a^{k},ax=xa,\text{ and } ax^2=x$, for some positive integer $k$. If the Drazin inverse of $a$ exists, then it is denoted by $a^d$. The smallest positive integer $k$ is called the {\it Drazin index}, is denoted by $i(a)$. The set of all Drazin invertible elements of $R$ will be denoted by $R^{d}.$  If $i(a) = 1$, then the Drazin inverse of $a$ is called as the  {\it group inverse} of $a$, and is denoted by $a^{\#}$. The set of group invertible elements of $R$ will be denoted by $R^{\#}$.

 The perusal of the core inverse is one of the areas of generalized inverses. That has caught the interest of numerous researchers in the past few decades. Firstly, Baksalary and Trenkler \cite{baks} introduced the core inverse for complex matrices. Let $A \in \mathbb{C}^{n\times n}$. A matrix $\core{A} \in \mathbb{C}^{n\times n}$ is called the {\it core inverse} of $A$ if
$$ A\core{A} = P_A \text { and }  R(\core{A})\subseteq R(A),$$
where $P_A$ is the orthogonal projector onto $R(A)$, and $R(A)$ is the column space of $A$.
Motivated by this work, Rakic {\it et al.} \cite{rakic} introduced core inverse in rings. An element $\core{a} \in R$ satisfying 
     $$a\core{a}a =a, \core{a}R = aR, \text{ and } R\core{a} = Ra^*,$$
is called the {\it core inverse} of $a$.  The authors proved that an element $a \in R$ is core invertible if and only if there exists $x \in R$ which satisfies
$$ axa = a,  xax = x,  (ax)^* = ax,  xa^2 = a, \text{and } ax^2 = x.$$
The element $x$ is known as the core inverse of $a.$ And it is unique (if exists). In 2016, Xu {\it et al}. \cite{Xu} proved that  if  $a$ is satisfying these three conditions  
$$(ax)^*=ax,~~ax^2=x,~ \text{ and }~ xa^2=a,$$
then $a$ is core invertible. The set of all core invertible elements of $R$ will be denoted by $\core{R}.$ In 2017, Mosic {\it et al.} \cite{mosicm} extended the notion of the core inverse to the weighted core inverse in a ring with involution. Let $a \in R$, and $e \in R$ be an invertible element with $e^* = e$. Then, a unique element $x \in R$ is said to be $e$-{\it weighted core inverse} if
$$ax^2 = x, xa^2 = a, \text{ and } (eax)^* = eax.$$
The $e$-weighted core inverse of an element $a\in R$  is denoted by $a^{\dc{}}$ (if exists). The set of all $e$-weighted core invertible elements of R is denoted by $R^{\dc{}}.$ These results will often be used later in this article.
 
\begin{lemma}(Corollary 3.4, \cite{zhou1})\label{t2.1}\\ Let $a, x \in R$ with $xa = ax$ and $xa^* = a^*x$. If $a \in \core{R}$, then $x\core{a} = \core{a}x$.
\end{lemma}
\begin{theorem}\label{thm1.5}  (Theorem 3.1, \cite{Chen14})\\
 Let $a, x \in R$. Then, the following are equivalent:\\
(i) $a \in \core{R}$, and $x = \core{a}$;\\
(ii) $axa = a, xR = aR$, and $Rx \subseteq Ra^*$;\\
(iii) $axa = a,$ ${^{\circ}x}$ = ${^{\circ}{a}}$, and $(a^*)^{\circ} \subseteq x^{\circ}$;\\
(iv) $xax = x, xR = aR$, and $Rx = Ra^*$;\\
(v) $xax = x, xR = aR$, and $Ra^* \subseteq Rx$;\\
(vi) $xax = x,$ ${^{\circ}{x}}$= ${^{\circ}a}$,  and $x^{\circ} \subseteq (a^*)^{\circ}$;\\
(vii) $a \in R^{\#}, axa = a, (ax)^* = ax$, and $xR \subseteq aR$;\\
(viii) $a \in R^{\#}, xax = x, (ax)^* = ax$, and $aR \subseteq xR$.
\end{theorem}

\section{Forward-order law for the core inverse}\label{sec2}
This section provides some sufficient conditions under which the forward-order law holds for the core inverse. First, we prove some results for the core inverse. We  discuss the forward-order law for the core inverse and the weighted core inverse. We then present some characterizations of the forward-order law for the core inverse. This section begins with the following lemma. 

\begin{lemma}\label{2.1}
Let $a\in \core{R}$. Then, the following conditions hold:
\begin{enumerate}[(i)]
    \item $(a^*)^{\circ}= (\core{a})^{	\circ};$
    \item  $^{\circ}[(\core{a})^*]= {^{\circ}(a)}= {^{\circ}(\core{a})};$ 
    \item  $\core{a}b=\core{a}c$ if and only if $a^*b=a^*c,$ where $b,c\in R;$
    \item  $b\core{a}=c\core{a}$ if and only if $ba=ca,$ where $b,c\in R;$
    \item  ${^{\circ}(b(\core{a})^*)}={^{\circ}(ba)}$, where $b\in R.$
\end{enumerate}

\end{lemma}

\begin{proof}
(i) We know that $$(a\core{a})^*=a\core{a},a(\core{a})^2=\core{a},\text{ and }  \core{a}a^2=a.$$
Let $x\in (a^*)^{\circ}.$ So, $a^*x=0$. If we pre-multiply $(\core{a})^*$ in $a^*x=0$, then we obtain $(\core{a})^*a^*x=0$, which implies
$a\core{a}x=0$. Pre-multiplying $\core{a}$ in equation $a\core{a}x=0$, we get $\core{a}a\core{a}x=0$, i.e., $\core{a}x=0$, i.e., $x\in (\core{a})^{\circ}.$
Therefore, 
\begin{equation}\label{equa2.1}
    (a^*)^{\circ}\subseteq(\core{a})^{\circ}.
\end{equation}
Conversely, if  $x\in (\core{a})^{\circ}$, then $\core{a}x=0$. Pre-multiplying $a$ in last equation $\core{a}x=0$, we have $ a\core{a}x=0$, i.e., $(a\core{a})^*x=0$. Again, pre-multiplying $a^*$ in $(a\core{a})^*x=0$, we get
    $a^*(\core{a})^*a^*x=0$, i.e., $a^*x=0$, i.e., $x\in(a^*)^{\circ}.$ Thus,
  \begin{equation}\label{equa2.2}
      (\core{a})^{\circ}\subseteq (a^*)^{\circ}.
  \end{equation} 
From \eqref{equa2.1} and \eqref{equa2.2}, we reach $(a^*)^{\circ}= (\core{a})^{\circ}$.\\
(ii) Let $x\in {^{\circ}[(\core{a})^*]}$. So, $x(\core{a})^*=0$. We know that $(a\core{a})^*=a\core{a}$. Post-multiplying  $a^*a$ in equation $x(\core{a})^*=0$, we obtain $x(a\core{a})^*a=0$, i.e., $xa\core{a}a=0$, i.e., $xa=0.$ Hence,  
\begin{equation}\label{equa2.3}
    {^{\circ}[(\core{a})^*]}\subseteq {^{\circ}(a)}.
\end{equation} Conversely, if $x\in {^{\circ}(a)},$ then $xa=0$. Post-multiplying  $\core{a}$ in $xa=0$, we have $xa\core{a}=0$, i.e., $x(\core{a})^*a^*=0$. Again, post-multiplying  $(\core{a})^*$, we get $x(\core{a})^*a^*(\core{a})^*=0$, i.e., $x(\core{a})^*=0.$ Thus,
\begin{equation}\label{equa2.4}
     ^{\circ}(a)\subseteq {^{\circ}[(\core{a})^*]}.
\end{equation} From \eqref{equa2.3} and \eqref{equa2.4}, we find  $^{\circ}[(\core{a})^*]= {^{\circ}(a)}$. Now, let $x\in {^{\circ}(a)}$, so $xa=0$. Post-multiplying  $(\core{a})^2$, we obtain
$xa({\core{a})^2}=0$, i.e., $x\core{a}=0.$ Therefore,
\begin{equation}\label{equa2.5}
    {^{\circ}(a)}\subseteq {^{\circ}(\core{a})}.
\end{equation} Conversely, we assume $x\in {^{\circ}(\core{a})}$. Then, post-multiplying $a^2$ in $x\core{a}=0$,  we have $x\core{a}a^2=0$, i.e., $xa=0$, i.e., $x\in {^{\circ}(a)}.$ Thus,
\begin{align}\label{equa2.6}
    {^{\circ}(\core{a})}\subseteq {^{\circ}(a)}.
\end{align} 
By \eqref{equa2.5} and \eqref{equa2.6}, we obtain  ${^{\circ}(\core{a})}= {^{\circ}(a)}$ implies $^{\circ}[(\core{a})^*]= {^{\circ}(a)}={^{\circ}(\core{a})}$.\\
(iii) We know that $(a\core{a})^*=a\core{a}$. We have $\core{a}b=\core{a}c$. Pre-multiplying $a^*a$, we get $a^*(\core{a})^*a^*b=a^*(\core{a})^*a^*c$, i.e.,
$a^*b=a^*c.$ Conversely, if  we pre-multiply  $\core{a}(\core{a})^*$ in $a^*b=a^*c$, then  we get $\core{a}b=\core{a}c.$\\
(iv) We have  $b\core{a}=c\core{a}$. Post-multiplying  $a^2$ in $b\core{a}=c\core{a}$, we obtain
$b\core{a}a^2=c\core{a}a^2$, i.e.,
   $ba=ca.$ Conversely, if we post-multiply by $(\core{a})^2$ in $ba=ca$, then we get $b\core{a}=c\core{a}.$\\
(v) Let $x\in {^{\circ}(b(\core{a})^*)}$. Then, we have $xb(\core{a})^*=0$. Post-multiplying  $a^*$ both sides of equation $xb(\core{a})^*=0$, we get $xb(\core{a})^*a^*=0$. So, $xb(a\core{a})^*=0$, i.e., $xba\core{a}=0$. Post-multiplying  $a$ both sides in $xba\core{a}=0$, and from $a\core{a}a=a$, we get $xba=0$. So, $x\in (ba)^{\circ}$, which implies ${^{\circ}(b(\core{a})^*)}\subseteq{^{\circ}(ba)}$. Conversely,  $x\in {^{\circ}}(ba)$ yields $xba=0$. Post-multiplying  $\core{a}$ both sides in $xba=0$, and from $(a\core{a})^*=(\core{a})^*a^*$, we get $xb(\core{a})^*a^*=0$. Again, post-multiplying  $(\core{a})^*$ both sides of equation $xb(\core{a})^*a^*=0$, and using the identity $(\core{a})^*a^*(\core{a})^*=(\core{a})^*$, we get $xb(\core{a})^*=0$. So, $x\in {^{\circ}(b(\core{a})^*)}$, which implies that ${^{\circ}(ba)}\subseteq{^{\circ}(b(\core{a})^*)}.$ Hence, ${^{\circ}(b(\core{a})^*)}={^{\circ}(ba)}$.
\end{proof}

Now, we present the forward-order law for the core inverse under the assumption of a few conditions.

\begin{theorem}\label{thm2.5}
Let $a,b\in \core{R}$ with $aba=ba^2=a^2b$. If $ab\core{a}=a\core{a}b$ and $ab\core{b}=b\core{b}a$, then $\core{(ab)}=\core{a}\core{b}.$
\end{theorem}
\begin{proof} Taking involution of  $ab\core{b}=b\core{b}a$, we have $a^*b\core{b}=b\core{b}a^*.$  By Lemma \ref{t2.1}, we obtain $\core{a}b\core{b}=b\core{b}\core{a}$. Now, we will prove that $\core{(ab)}=\core{a}\core{b}$  by using the definition of the core inverse.
\begin{align*}
    ab\core{a}\core{b}\core{a}\core{b}&=a\core{a}b\core{b}\core{a}\core{b}\\
    &=a\core{a}\core{a}b\core{b}\core{b}\\
    &=\core{a}\core{b},
\end{align*}
\begin{align*}
    \core{a}\core{b}abab&=\core{a}\core{b}ba^2b\\
    &=\core{a}\core{b}b^2a^2\\
    &=\core{a}ba^{2}\\
    &=\core{a}a^2b\\
    &=ab,
\end{align*}
and
\begin{align*}
    (ab\core{a}\core{b})^*&=(a\core{a}b\core{b})^*\\
    &=(ab\core{b}\core{a})^*\\
    &=(b\core{b}a\core{a})^*\\
    &=(a\core{a})^*(b\core{b})^*\\
    &=a\core{a}b\core{b}\\
    &=ab\core{a}\core{b}.
\end{align*}
Hence, $\core{(ab)}=\core{a}\core{b}.$
\end{proof}
 Note that the condition  $aba=ba^2=a^2b$  in the above result can be replaced by $ab=ba$.
The next result provides another set of sufficient conditions for the forward-order law. 

\begin{theorem}\label{th5.2}
Let $a, b\in\core{R}$. If $a^*b=a^*a\core{(ab)}b^2$ and $\core{b}ba=ab\core{b}$, then $\core{(ab)}=\core{a}\core{b}$.
\end{theorem}
\begin{proof}

From Lemma \ref{2.1} (iii), $a^*b=a^*a\core{(ab)}b^2$ yields  $\core{a}b=\core{a}a\core{(ab)}b^2$. From $ab=ab\core{b}b$, and  $\core{b}ba=ab\core{b}$, we get $ab=b\core{b}ab$ which implies that $(ab)^*=(ab)^*b\core{b}.$ From Lemma \ref{2.1} (iii), $\core{(ab)}=\core{(ab)}b\core{b}.$ Now, 
\begin{align*}
    \core{a}\core{b}=\core{a}a\core{a}b(\core{b})^2&=\core{a}(a\core{a})^*b(\core{b})^2=\core{a}(\core{a})^*a^*b(\core{b})^2\\
    &=\core{a}(\core{a})^*a^*a\core{(ab)}b^2(\core{b})^2=\core{a}a\core{(ab)}b\core{b}=\core{a}a\core{(ab)}\\
    &=\core{a}aab\core{(ab)}\core{(ab)}=ab\core{(ab)}\core{(ab)}=\core{(ab)}.
\end{align*}
\end{proof}

A characterization of the forward-order law for a class of elements  satisfing the condition  $b\core{b}a=ab\core{b}$ is established below.
\begin{theorem}\label{th2.6}
Let $a,b\in\core{R}$ with $b\core{b}a=ab\core{b}$. Then, the following are equivalent:
\begin{enumerate}[(i)]
    \item $ab\in\core{R}$ and $\core{(ab)}=\core{a}\core{b}$;
    \item $ab\in R^{\#}$, $a\core{b}R\subseteq \core{a}\core{b}R$, and  $a^*\core{b}(1-(ab\core{a}\core{b})^*)b=0$.
\end{enumerate}
\end{theorem}
\begin{proof}
(i)$\Rightarrow$(ii):
 From Theorem \ref{thm1.5} (viii), $ab\in R^{\#}.$ Now, $a\core{b}R\subseteq ab(\core{b})^2R\\
 \subseteq \core{(ab)}(ab)^2(\core{b})^2R\subseteq \core{(ab)}R\subseteq \core{a}\core{b}R.$ So, $a\core{b}R\subseteq \core{a}\core{b}R$.  Further, we obtain
\begin{align}\label{equ2.1}
    a^*\core{b}b&=(a\core{a}a)^*\core{b}b\nonumber\\
    &=a^*a\core{a}\core{b}b\nonumber\\
    &=a^*a\core{(ab)}b\nonumber\\
    &=a^*a\core{(ab)}ab\core{(ab)}b\nonumber\\
    &=a^*(a\core{a})^*\core{b}ab\core{a}\core{b}b\nonumber\\
   &=a^*(\core{a})^*a^*\core{b}ab\core{a}\core{b}b\nonumber\\
   &=a^*\core{b}ab\core{a}\core{b}b\nonumber\\
   &=a^*\core{b}ab\core{(ab)}b\nonumber\\
   &=a^*\core{b}(ab\core{(ab)})^*b\nonumber\\
   &=a^*\core{b}(ab\core{a}\core{b})^*b.
  \end{align}
  Equation \eqref{equ2.1} implies that 
 $a^*\core{b}(1-(ab\core{a}\core{b})^*)b=0$.\\
 (ii)$\Rightarrow$(i):  We will show that  $\core{(ab)}=\core{a}\core{b}$ by using Theorem \ref{thm1.5} (viii). Setting $x=\core{a}\core{b}$. From  part (ii), we have   $a^*\core{b}(1-(ab\core{a}\core{b})^*)b=0$, i.e.,  $a^*\core{b}b=a^*\core{b}(ab\core{a}\core{b})^*b.$ We can write, $a^*\core{b}=(a^*\core{b}b)\core{b}$, which implies that $a^*\core{b}=a^*\core{b}(ab\core{a}\core{b})^*b\core{b}$. Taking involution both sides in last equality $a^*\core{b}=a^*\core{b}(ab\core{a}\core{b})^*b\core{b}$, we get
\begin{align}\label{eq2.2}
    (\core{b})^*a&=(b\core{b})^*ab\core{a}\core{b}(\core{b})^*a.
    \end{align}
    Post-multiplying  $\core{a}(\core{a})^*$ in \eqref{eq2.2}, we obtain
    \begin{align}\label{eq2.3}
    (\core{b})^*(\core{a})^*&=(b\core{b})^*ab\core{a}\core{b}(\core{b})^*(\core{a})^*,\nonumber
    \end{align}
    i.e.,
    \begin{align}
    (\core{a}\core{b})^*&=(b\core{b})^*ab\core{a}\core{b}(\core{a}\core{b})^*.
    \end{align}
    Again, taking involution both sides in \eqref{eq2.3}, we get
    \begin{align}\label{eq2.4}
  \core{a}\core{b}&=\core{a}\core{b}(ab\core{a}\core{b})^*(b\core{b})^*\nonumber\\
   &=\core{a}\core{b}(b\core{b}ab\core{a}\core{b})^*\nonumber\\
      &=\core{a}\core{b}(ab\core{b}b\core{a}\core{b})^*\nonumber\\
      &=\core{a}\core{b}(ab\core{a}\core{b})^*.
      \end{align}
      Pre-multiplying $ab$ in \eqref{eq2.4}, we obtain
      \begin{align}\label{eq2.5}
   ab\core{a}\core{b}&=ab\core{a}\core{b}(ab\core{a}\core{b})^*.
\end{align}
By equation \eqref{eq2.5}, we have $(ab\core{a}\core{b})^*=ab\core{a}\core{b}$, i.e., $(abx)^*=abx$. The previous equality $(ab\core{a}\core{b})^*=ab\core{a}\core{b}$ and equation \eqref{eq2.4} imply that $\core{a}\core{b}ab\core{a}\core{b}=\core{a}\core{b}$, i.e., $xabx=x$. From  $a\core{b}R\subseteq \core{a}\core{b}R=xR$, we get $a\core{b}=\core{a}\core{b}y,$ where $y\in R$. Then, $ab=a\core{b}b^2=\core{a}\core{b}yb^2$. And $abR\subseteq \core{a}\core{b}yb^2R\subseteq \core{a}\core{b}R$, i.e., $abR\subseteq \core{a}\core{b}R=xR$. By Theorem \ref{thm1.5} (viii), we thus have  $ab\in \core{R}$ and $\core{(ab)}=x=\core{a}\core{b}.$ 
\end{proof}
The next result provides sufficient conditions for which the set of core invertible elements satisfies the commutative property. 
\begin{theorem}\label{th2.7}
Let $a,b\in \core{R}$ with $a^*\core{b}=\core{b}a^*$ and $bab=ab^2$. If $a\core{b}=\core{b}a$ and $\core{a}b=b\core{a}$, then $\core{(ab)}=\core{a}\core{b}=\core{b}\core{a}.$
\end{theorem}
\begin{proof}
First we will show that  $\core{(ab)}=\core{a}\core{b}$ by using Theorem \ref{thm1.5} (ii). Setting $x=\core{a}\core{b}$, we obtain
\begin{align*}
    abxab=ab\core{a}\core{b}ab&=a\core{a}b\core{b}ab\\
    &=a\core{a}ba\core{b}b\\
    &=a\core{a}bab(\core{b})^2b\\
    &=a\core{a}ab^2(\core{b})^2b\\
    &=ab\core{b}b\\
    &=ab.
\end{align*}
So, the first condition of Theorem \ref{thm1.5} (ii) is satisfied. Next to show that $xR = abR$ which is the second condition of Theorem \ref{thm1.5} (ii). By hypothesis, we have $\core{a}b=b\core{a}$. Thus     $$\core{a}\core{b}R\subseteq a(\core{a})^2b(\core{b})^2R\subseteq ab(\core{a})^2(\core{b})^2R\subseteq abR, ~\mbox{i.e.,}~ \core{a}\core{b}R\subseteq abR.$$
  Hence, $xR \subseteq abR$.  Conversely, we have $\core{b}a=a\core{b}$. So, $abR\subseteq \core{a}a^2\core{b}b^2R\subseteq \core{a}\core{b}a^2b^2R\subseteq \core{a}\core{b}R,$ i.e, $abR\subseteq \core{a}\core{b}R$. So, $abR\subseteq xR$.  Hence,
    $abR=xR.$
    One can now apply  Theorem \ref{thm1.5} (ii) if  the third condition $Rx \subseteq R(ab)^*$ holds. We have 
    $Rx=R\core{a}\core{b}\subseteq R\core{a}a\core{a}\core{b}\subseteq R(a\core{a})^*\core{b}\subseteq R(\core{a})^*a^*\core{b}\subseteq Ra^*\core{b}\subseteq R\core{b}a^*
    \subseteq R\core{b}b\core{b}a^*\subseteq R(b\core{b})^*a^*\subseteq R(\core{b})^*b^*a^*\subseteq R(ab)^*$, i.e., $Rx=R\core{a}\core{b}\subseteq R(ab)^*.$
 Hence,  Theorem \ref{thm1.5} (ii) yields $$\core{(ab)}=\core{a}\core{b}.$$ We also have $a^*\core{b}=\core{b}a^*$ and $a\core{b}=\core{b}a$.  By Lemma \ref{t2.1}, we obtain $\core{a}\core{b}=\core{b}\core{a}.$ Thus, $$\core{(ab)}=\core{a}\core{b}=\core{b}\core{a}.$$
\end{proof}

\noindent  For $a,b\in \core{R}$, one can show that the conditions  $ab=ba$ and $ab^*=b^*a$ imply $a^*\core{b}=\core{b}a^*$, $bab=ab^2$, $a\core{b}=\core{b}a$ and $\core{a}b=b\core{a}$. The next result replaces the four conditions by the above mentioned two conditions, and can be proved  similarly as the above one.
\begin{remark}\label{t2.20}
Let $a,b\in \core{R}$. If $ab=ba$ and $ab^*=b^*a$, then $\core{(ab)}=\core{a}\core{b}=\core{b}\core{a}$.
\end{remark}
\noindent 
Kumar and Mishra \cite{Kumar2} proposed the following result for idempotent elements.

\begin{theorem}\label{l2.2}(Theorem 2.11, \cite{Kumar2})\\
Let $a,b \in R$, and $x,y\in R$ be two idempotent elements. Then, the following hold:
\begin{enumerate}[(i)]
    \item  $(1-x)a=b$ if and only if $xb=0$ and ${^{\circ}(x)}\subset{^{\circ}(a-b)};$
    \item $a(1-y)=b$ if and only if $by=0$ and $ (y)^{\circ}\subset (a-b)^{\circ}.$
\end{enumerate}
\end{theorem}
\noindent  If $b\in \core{R}$, then $b\core{b}$ is an idempotent element. By Theorem \ref{l2.2},  we thus have the following remark.
\begin{remark}
Let $a,b\in \core{R}$ with $ab=b\core{b}ab$. Then  ${^{\circ}}(b\core{b})\subseteq {^{\circ}}(ab)$.
\end{remark}
\noindent A characterization  of the forward-order law is presented below.
\begin{theorem}
Let $a,b,ab\in\core{R}$. Then,  $\core{(ab)}={\core{a}\core{b}}$ if and only if
     $a^*a\core{(ab)}=a^*\core{b}.$
    \end{theorem}
\begin{proof}
     We know that $(a\core{a})^*=(\core{a})^*a^*$. Now, pre-multiplying $a^*a$ in $\core{(ab)}=\core{a}\core{b}$, we obtain $a^*a\core{(ab)}=a^*(\core{a})^*a^*\core{b}$ which implies $a^*a\core{(ab)}=a^*\core{b}.$ Conversely,
  pre-multiplying $(\core{a})^*$ in $a^*a\core{(ab)}=a^*\core{b}$, we get $a\core{(ab)}=a\core{a}\core{b}.$ Again, pre-multiplying $\core{a}$ in $a\core{(ab)}=a\core{a}\core{b}$, we get $\core{a}a\core{(ab)}=\core{a}\core{b}.$ Further, we have
    $$\core{a}a\core{(ab)}=\core{a}a(ab)\core{(ab)}\core{(ab)}=ab\core{(ab)}\core{(ab)}=\core{(ab)}.$$
    Hence, $\core{(ab)}=\core{a}\core{b}.$
    
\end{proof}
We now show that the forward-order law for the weighted core inverse holds under the assumption $ab=b^2$.

\begin{theorem}
Let $a,b\in R^{\dc{}}$. If $ab=b^2$, then
\begin{enumerate}[(i)]
    \item $ab\in R^{\dc{}}$ and $(ab)^{\dc{}}=a^{\dc{}}b^{\dc{}}$;
    \item $aa^{\dc{}}b^{\dc{}}(baa^{\dc{}})^2=baa^{\dc{}}$,  $ebaa^{\dc{}}(aa^{\dc{}}b^{\dc{}})^2=eb^{\dc{}}$ and ${(ebaa^{\dc{}}b^{\dc{}})^*}=ebaa^{\dc{}}b^{\dc{}}.$
\end{enumerate}
\end{theorem}
\begin{proof}
\begin{enumerate}[(i)]
    \item
By the group inverse definition, we can write $b=b^2b^{\#}$. The hypothesis $ab=b^2$ and  $b=b^2b^{\#}$ imply that $b=abb^{\#}$, i.e., $b=a^{\dc{}}a^2bb^{\#}$, i.e., $b=a^{\dc{}}b^3b^{\#}$. So, $b=a^{\dc{}}b^2$. Again, by $ab=b^2$, we get
$b^{\dc{}}b=(b^{\dc{}})^2b^2=(b^{\dc{}})^2ab$. Similarly, $bb^{\dc{}}=b^2(b^{\dc{}})^2=ab(b^{\dc{}})^2=ab^{\dc{}}$. Further, we get
\begin{align*}
    a^{\dc{}}b^{\dc{}}abab&=a^{\dc{}}b^{\dc{}}b^2b^2\\
   & =a^{\dc{}}bb^2\\
    &=b^2\\
    &=ab,  
\end{align*}
\begin{align*}
    aba^{\dc{}}b^{\dc{}}a^{\dc{}}b^{\dc{}}&=aba^{\dc{}}b^{\dc{}}a^{\dc{}}b^2(b^{\dc{}})^3\\
&=aba^{\dc{}}b^{\dc{}}b(b^{\dc{}})^3\\
&=aba^{\dc{}}(b^{\dc{}})^3\\
&=a^{\dc{}}a^2ba^{\dc{}}(b^{\dc{}})^3\\
&=a^{\dc{}}b^3a^{\dc{}}(b^{\dc{}})^3\\
&=a^{\dc{}}b^3a^{\dc{}}b^2(b^{\dc{}})^5\\
&=a^{\dc{}}b^4(b^{\dc{}})^5\\
&=a^{\dc{}}b^{\dc{}},
\end{align*}
and
\begin{align*}
    (eaba^{\dc{}}b^{\dc{}})^*&=(eb^2a^{\dc{}}b^2(b^{\dc{}})^3)^*\\
    &=(eb^3(b^{\dc{}})^3)^*
       \end{align*}
\begin{align*}
    &=(ebb^{\dc{}})^*\\
    &=ebb^{\dc{}}\\
    &=eaba^{\dc{}}b^{\dc{}}.
\end{align*}
Hence, $(ab)^{\dc{}}=a^{\dc{}}b^{\dc{}}.$
\item Putting $a^{\dc{}}b^{\dc{}}$ directly in equations, we have
\begin{align}\label{kul}
    ebaa^{\dc{}}b^{\dc{}}
    &=ebaa^{\dc{}}b^2(b^{\dc{}})^3\nonumber\\
    &=ebaa^{\dc{}}ab(b^{\dc{}})^3\nonumber\\
    &=ebab(b^{\dc{}})^3\nonumber\\
    &=eb^3(b^{\dc{}})^3\nonumber\\
    &=ebb^{\dc{}},
\end{align}
which implies $(ebaa^{\dc{}}b^{\dc{}})^*=ebaa^{\dc{}}b^{\dc{}}$. Further, we obtain
\begin{align*}
    aa^{\dc{}}b^{\dc{}}(baa^{\dc{}})^2&=aa^{\dc{}}b^{\dc{}}baa^{\dc{}}baa^{\dc{}}\\
    &= aa^{\dc{}}b^2(b^{\dc{}})^3baa^{\dc{}}baa^{\dc{}}\\
    &=aa^{\dc{}}ab(b^{\dc{}})^3baa^{\dc{}}baa^{\dc{}}\\
    &=ab(b^{\dc{}})^3baa^{\dc{}}baa^{\dc{}}\\
    &=b^2(b^{\dc{}})^3baa^{\dc{}}baa^{\dc{}}\\
    &=b^{\dc{}}baa^{\dc{}}b^{\dc{}}b^2aa^{\dc{}}\\
    &=b^{\dc{}}baa^{\dc{}}b^2(b^{\dc{}})^3b^2aa^{\dc{}}\\
    &=b^{\dc{}}baa^{\dc{}}ab(b^{\dc{}})^3b^2aa^{\dc{}}\\
    &=b^{\dc{}}bab(b^{\dc{}})^3b^2aa^{\dc{}}\\
    &=b^{\dc{}}bb^2(b^{\dc{}})^3b^2aa^{\dc{}}\\
    &=b^2(b^{\dc{}})^3b^2aa^{\dc{}}\\
    &=b(b^{\dc{}})^2b^2aa^{\dc{}}\\
    &=baa^{\dc{}},
\end{align*}
and
\begin{align*}
    ebaa^{\dc{}}(aa^{\dc{}}b^{\dc{}})^2&=ebaa^{\dc{}}aa^{\dc{}}b^{\dc{}}(aa^{\dc{}}b^{\dc{}})\\
    &=ebaa^{\dc{}}b^{\dc{}}(aa^{\dc{}}b^{\dc{}}) \\
    &=ebb^{\dc{}}(aa^{\dc{}}b^{\dc{}})\text{ (By equation $\eqref{kul}$)}\\
    &=ebb^{\dc{}}aa^{\dc{}}b^2(b^{\dc{}})^3\\
    &=ebb^{\dc{}}aa^{\dc{}}ab(b^{\dc{}})^3
       \end{align*}
\begin{align*}
  &=ebb^{\dc{}}ab(b^{\dc{}})^3\\
      &=ebb^{\dc{}}b^2(b^{\dc{}})^3\\
       &=eb^{\dc{}}.
\end{align*}
\end{enumerate}
\end{proof}

\begin{theorem}\label{thm3.3}
Let $a,b\in \core{R}$. Then, 
\begin{enumerate}[(i)]
    \item $(ab)^{\dagger}=\core{a}\core{b}$ if and only if $abR\subseteq bR$ and $\core{a}b=(ab)^{\dagger}b^2;$
    \item $(ab)^{\dagger}=\core{a}\core{b}$ if and only if $Rab\subseteq Ra^*$ and $a^*\core{b}=a^*a(ab)^{\dagger};$
\end{enumerate}
 \end{theorem}
\begin{proof}
\begin{enumerate}[(i)]
\item  $(ab)^{\dagger}=\core{a}\core{b}$ yields $abR\subseteq bR$. Further, $(ab)^{\dagger}b^2=\core{a}\core{b}b^2=\core{a}b.$ Conversely,  $abR\subseteq bR$ implies $ab=bz$, where $z\in R$. Pre-multiplying $b\core{b}$ in $ab=bz$, we get  $b\core{b}ab=ab$. Taking involution of $ab=b\core{b}ab$, we get $(ab)^*=(ab)^*b\core{b}$. Pre-multiplying $(ab)^{\dagger}((ab)^{\dagger})^*$ in $(ab)^*=(ab)^*b\core{b}$, we have $(ab)^{\dagger}=(ab)^{\dagger}b\core{b}.$ Post-multiplying  $(\core{b})^2$ in $\core{a}b=(ab)^{\dagger}b^2$, we get $\core{a}\core{b}=(ab)^{\dagger}b\core{b}$. Hence, $(ab)^{\dagger}=\core{a}\core{b}.$
 \item If $(ab)^{\dagger}=\core{a}\core{b}$, then  $abR\subseteq Ra^*$. Further, $a^*a(ab)^{\dagger}=a^*a\core{a}\core{b}=a^*(\core{a})^*a^*\core{b}=a^*\core{b}$, i.e., $a^*\core{b}=a^*a(ab)^{\dagger}.$
 Conversely, $Rab\subseteq Ra^*$ implies $ab=za^*$, where $z\in R.$ Post-multiplying  $(\core{a}a)^*$, we get  $ab=ab(\core{a}a)^*$. Taking involution of $ab=ab(\core{a}a)^*$, we get $(ab)^*=\core{a}a(ab)^*$. Again, post-multiplying  $((ab)^{\dagger})^*(ab)^{\dagger}$, we obtain $(ab)^{\dagger}=\core{a}a(ab)^{\dagger}$. Pre-multiplying $\core{a}(\core{a})^*$ in $a^*\core{b}=a^*a(ab)^{\dagger}$, we get $\core{a}\core{b}=\core{a}a(ab)^{\dagger}$. Hence, $(ab)^{\dagger}=\core{a}\core{b}.$

\end{enumerate}
\end{proof}

\section{Conclusion}
The important findings are summarized as follows:
\begin{itemize}
\item The forward-order laws for the core inverse and the weighted core inverse have been introduced in rings.
\item Finally, we have presented a few necessary and sufficient conditions of the hybrid forward-order law.
\end{itemize}

\section{Acknowledgements}
The first author acknowledges the support of the Council of Scientific and Industrial Research, India. We thank Aaisha Be and Vaibhav Shekhar  for their helpful suggestions on some parts of this article.

\end{document}